\documentclass[pdftex]{amsart}

\usepackage{amssymb}
\usepackage{enumerate}
\usepackage{rotating}
\usepackage{mathrsfs}
\usepackage[matrix,arrow,graph,curve]{xy}

\theoremstyle{plain}      
    \newtheorem{theorem}{Theorem}[section]
    \newtheorem{proposition}[theorem]{Proposition}
    
    \newtheorem{corollary}[theorem]{Corollary}
\theoremstyle{definition}
    \newtheorem{definition}[theorem]{Definition}
    
\theoremstyle{remark}
    \newtheorem*{remark}{Remark}

\newcommand{\A}{\ensuremath{\mathscr{A}}}
\newcommand{\C}{\ensuremath{\mathscr{C}}}

\newcommand{\Set}{\ensuremath{\mathbf{Set}}}

\renewcommand{\phi}{\ensuremath{\varphi}}
\newcommand{\ox}{\ensuremath{\otimes}}
\newcommand{\op}{\ensuremath{\mathrm{op}}}
\renewcommand{\o}{\ensuremath{\circ}}
\newcommand{\lraise}[2]{\ensuremath{{}^{#1} \! {#2}}}
\newcommand{\MN}{\ensuremath{{}^M \! N}}
\newcommand{\ra}{\ensuremath{\xymatrix@1@=22pt{\ar[r]&}}}
\newcommand{\riso}{\ensuremath{\xymatrix@1@=27pt{\ar[r]^-\cong&}}}

\begin{document}
\title{Closed categories, star-autonomy, and monoidal comonads}
\author{Craig Pastro}
\author{Ross Street}
\thanks{The first author gratefully acknowledges support of an international
Macquarie University Research Scholarship while the second gratefully
acknowledges support of the Australian Research Council Discovery Grant
DP0771252.}
\address{Department of Mathematics \\
         Macquarie University \\
         New South Wales 2109 Australia}
\email{\{craig, street\}@ics.mq.edu.au}
\date{\today}
\dedicatory{Dedicated to Gus Lehrer on his 60th birthday}
\begin{abstract}
This paper determines what structure is needed for internal homs in a
monoidal category $\C$ to be liftable to the category $\C^G$ of
Eilenberg-Moore coalgebras for a monoidal comonad $G$ on $\C$. We apply
this to lift $*$-autonomy with the view to recasting the definition of
quantum groupoid.
\end{abstract}
\maketitle

\section*{Introduction}

It was recognized by Szlach\'anyi~\cite{Sz03} that Takeuchi's
$\times_R$-bialgebras (bialgebroids) could be described as opmonoidal
monads. Brzezi\'nski and Militaru~\cite{BM02} developed this further and
dualized the notion. The dual concept was called quantum category
in~\cite{DS} and was expressed in terms of a monoidal comonad; however
the main point of the paper was to obtain a definition of quantum groupoid
which involved $*$-autonomy in the sense of Barr~\cite{B}. This
$*$-autonomy amounts to an antipode in the case of a bialgebra (which is
a ``one object'' quantum category). The paper~\cite{DS} expressed the
generalized antipode as a structure on a generating monoidal adjunction
(``basic data'') for the comonad, rather than giving this antipode in terms
of the monoidal comonad itself. Motivation for the present paper was to
clarify the latter possibility.

The problem leads to one that can be stated for monads $T$ on ordinary
monoidal categories $\A$. It was pointed out in~\cite{M} that the
category $\A^T$ of Eilenberg-Moore algebras for an opmonoidal monad $T$
becomes monoidal in such a way that the underlying functor $U:\A^T \ra
\A$ becomes strict monoidal. We ask when internal homs in $\A$ can be lifted
to $\A^T$. More specifically, we ask under what extra structure on $T$ does
the Eilenberg-Moore category $\A^T$ become $*$-autonomous if $\A$ is.

In the meantime, the paper~\cite{BV} came to our notice, solving the
autonomous case. An autonomous category in the sense of~\cite{JS} (also,
well before that, called ``compact'' and ``rigid'' in the symmetric case)
admits a left and right dual for each object. A common generalization of
antipode for a bialgebra and autonomy for a monoidal category was obtained
in~\cite{DMS} and called ``dualization''. The concept of antipode $\nu$ for
an opmonoidal monad $T$ on an autonomous monoidal category $\A$ is defined
in~\cite{BV} and the pair $(T,\nu)$ is there called a ``Hopf monad''.
Autonomy is a special case of $*$-autonomy so~\cite{BV} answered our
questions in an important special case.

Our present paper answers the question of lifting $*$-autonomy. Our
motivation from quantum groupoids causes us to write in terms of a monoidal
comonad $G$ on a monoidal category $\C$ rather than an opmonoidal monad $T$.
Since we are interested in abstracting our work to monoidal comonads in
monoidal bicategories, this duality is not a serious point of difference.
There are some new subtleties required in the non-autonomous case arising
from the lack of unit and counit morphisms involved with duals in $\C$;
we must be content with the coevaluation and evaluation morphisms associated
with the weaker duality of $*$-autonomy.

In Section~\ref{sec-raisers} we review closed and $*$-autonomous categories
from the point of view of what we are calling ``raisers''. Section~\ref{sec-mc}
reviews monoidal comonads and describes what is required to lift a raiser
from a category $\C$ to the category $\C^G$ of Eilenberg-Moore
$G$-coalgebras for a monoidal comonad $G$. In Section~\ref{sec-starmc} we
define what it means for a monoidal comonad $G$ to be (left) $*$-autonomous
and prove the main result of our paper, viz., that $\C^G$ is (left)
$*$-autonomous if $G$ is. Section~\ref{sec-moncom} starts from a monoidal
adjunction and investigates what is required on the adjunction to reproduce
the results of Section~\ref{sec-starmc} for the induced comonad. In
Section~\ref{sec-hopfeg} we show that a Hopf algebra in a braided
$*$-autonomous category gives an example of a $*$-autonomous comonad.

We would like to thank Brian Day and Steve Lack for several helpful
suggestions.

\section{Internal homs and raisers} \label{sec-raisers}

Let $D$ be an object of a monoidal category $\C$. A \emph{left internal hom}
for objects $B$ and $D$ is a representing object $\lraise{B}{D}$ (or
$[B,D]_l$) for the functor
\[
    \C(- \ox B,D):\C^\op \ra \Set.
\]
This means that the object $\lraise{B}{D}$ comes equipped with an isomorphism
\[
    \varpi_{A,B}:\C(A,\lraise{B}{D}) \riso \C(A \ox B,D)
\]
which is natural in $A \in \C$. By taking $A = \lraise{B}{D}$ and evaluating
at the identity, we obtain an \emph{evaluation morphism}
\[
    e_B:\lraise{B}{D} \ox B \ra D.
\]
By the Yoneda lemma, $\varpi_{A,B}$ is recaptured as the composite
\[
\xygraph{{\C(A,\lraise{B}{D})}
    :[r(2.8)] {\C(A \ox B,\lraise{B}{D} \ox B)} ^-{- \ox B}
    :[r(3.2)] {\C(A \ox B,D)} ^-{\C(1,e_B)}}.
\]

For $B = I$, the unit for $\ox$ on $\C$, we always have the choice
$\lraise{I}{D} = D$ with $e_I:D \ox I \ra D$ equal to the right unit
isomorphism.

Our object $D$ is called a \emph{left raiser} when there is a choice of
$\lraise{B}{D}$ for all $B \in \C$. Again by Yoneda, we obtain a unique
functor
\[
    S = \lraise{-}{D}:\C^\op \ra \C
\]
defined on objects by $SB = \lraise{B}{D}$ and such that $\varpi_{A,B}$ becomes
natural in $B \in \C$. This last is equivalent to saying that $e_B$ is
natural in $B$ in the sense of Eilenberg-Kelly~\cite{EK}. We can easily
modify the tensor product to make the unit $I$ strict, so we can ensure that
\[
    SI = D, \quad \varpi_{A,I} = 1_{\C(A,D)}, \quad \text{and} \quad e_I = 1_D.
\]

The composite natural isomorphism
\begin{multline*}
\xygraph{{\C(A,S(B \ox C))}
    :[r(3.6)] {\C(A \ox (B \ox C),D)} ^-{\varpi_{A,B \ox C}}} \\
\xygraph{:[r(2.2)] {\C((A \ox B) \ox C,D)} ^-\cong
    :[r(3.4)] {\C(A \ox B, SC)} ^-{\varpi^{-1}_{A \ox B,C}}}
\end{multline*}
will be denoted by
\[
    \omega_{A,B,C}:\C(A,S(B \ox C)) \riso \C(A \ox B,SC)~.
\]
It follows that, if $D$ is a left raiser, so too is $SC = \lraise{C}{D}$ with
$\lraise{B}{SC} = S(B \ox C)$. Note also that $\omega_{A,I,C} = 1_{\C(A,SC)}$
and $\omega_{A,B,I} = \varpi_{A,B}$.

Assume that $D$ is a left raiser for the remainder of this section.

By taking $A = S(B \ox C)$ and evaluating at the identity, the isomorphism
$\omega_{A,B,C}$ defines a morphism
\[
    e_{B,C}:S(B \ox C) \ox B \ra SC.
\]
By Yoneda, $\omega_{A,B,C}$ is recovered as the composite
\[
\xygraph{{\C(A,S(B \ox C))}
    :[r(3.6)] {\C(A \ox B,S(B \ox C) \ox B)} ^-{- \ox B}
    :[r(3.9)] {\C(A \ox B,SC)} ^-{\C(1,e_{B,C})}}.
\]
In particular, $e_{I,C} = 1_{SC}$. From the definition of $\omega$ in terms
of $\varpi$ and $\varpi^{-1}$, we obtain the commutativity of the
triangle\footnote{We are now writing as if $\C$ were strict monoidal,
however this is not necessary.}
\[
\vcenter{\xygraph{{S(B \ox C) \ox B \ox C}="0"
    :[r(3)] {SC \ox C} ^-{e_{B,C} \ox 1}
    :[l(1.4)d] {D.}="1" ^-{e_C}
 "0":"1" _{e_{B \ox C}}}}
\]
In particular, $e_{B,I} = e_B:SB \ox B \ra D$.

We define a natural isomorphism
\[
    \rho_{A,B} = \omega^{-1}_{I,A,B}:\C(A,SB) \riso \C(I,S(A \ox B))~.
\]
Taking $A = SB$ and evaluating at the identity, we obtain a morphism
\[
    n_B : I \ra S(SB \ox B)
\]
natural in $B$. By Yoneda, $\rho_{A,B}$ is the composite
\[
\xygraph{{\C(A,SB)}
    :[r(3.6)] {\C(S(SB \ox B),S(A \ox B))} ^-{S(- \ox B)}
    :[r(4)] {\C(I,S(A \ox B))} ^-{\C(n_B,1)}}.
\]
Using the formula for $\omega_{I,A,B}$ in terms of $e_{A,B}$, we obtain the
commutativity of the triangle
\[
\vcenter{\xygraph{{SB}="s"
    :[r(2.4)] {S(SB \ox B) \ox SB} ^-{n_B \ox 1}
    :[l(1.2)d] {SB.}="t" ^-{e_{SB,B}}
 "s":"t" _-{1_{SB}}}}
\]

\begin{proposition}\label{prop1}
The following triangle commutes.
\[
\xygraph{{\C(I,S(A \ox B \ox C))}="s"
    :[r(3.8)] {\C(A,S(B \ox C))} ^-{\omega_{I,A,B \ox C}}
    :[l(1.9)d(1.2)] {\C(A \ox B,SC)}="t" ^-{\omega_{A,B,C}}
 "s":"t" _-{\omega_{I,A \ox B,C}}}
\]
\end{proposition}

\begin{proof}
This is verified by the following calculation.
\begin{align*}
\omega_{A,B,C}~\omega_{I,A,B \ox C}
    &= \varpi^{-1}_{A \ox B,C}~\varpi_{A,B \ox C}~\varpi^{-1}_{A,B \ox C}~
    \varpi_{I,A \ox B \ox C} \\
    &= \varpi^{-1}_{A \ox B,C}~\varpi_{I,A \ox B \ox C} \\
    &= \omega_{I,A \ox B,C}
\end{align*}
\end{proof}

\begin{corollary}\label{cor2}
The following triangles commute.
\begin{enumerate}[{\upshape (i)}]
\item
\[
\xygraph{{\C(A,S(B \ox C))}="s"
    :[r(3.2)] {\C(A \ox B,SC)} ^-{\omega_{A,B,C}}
    :[l(1.5)d(1.2)] {\C(I,S(A \ox B \ox C))}="t" ^-{\rho_{A \ox B,C}}
 "s":"t" _-{\rho_{A,B \ox C}}}
\]

\item 
\[
\xygraph{{I}="s"
    :[r(2)] {S(SB \ox B)} ^-{n_B} 
    :[d(1.1)l(0.7)] {S(S(A \ox B) \ox A \ox B)}="t" ^-{S(e_{A,B})}
 "s":"t" _-{n_{A \ox B}}}
\]

\item 
\[
\xygraph{{I}="s"
    :[r(1.6)] {SSI} ^-{n_I}
    :[l(0.8)d(1)] {S(SA \ox A)}="t" ^-{S(e_A)}
 "s":"t" _-{n_A}}
\]
\end{enumerate}
\end{corollary}

\begin{proposition}\label{prop3}
The inverse of $\varpi_{I,B}$ is the composite
\[
\xygraph{{\C(B,SI)}
    :[r(2.2)] {\C(SSI,SB)} ^-S
    :[r(2.6)] {\C(I,SB)} ^-{\C(n_I,1)}}.
\]
\end{proposition}

\begin{proof}
$\varpi_{I,B} = \omega_{I,B,I}$ has inverse $\rho_{B,I}$ and this composite
is the formula for $\rho_{B,I}$ in terms of $n_I$.
\end{proof}

\begin{corollary}
\begin{enumerate}[{\upshape (i)}]
\item The inverse of $\varpi_{A,B}$ is the composite
\begin{multline*}
\xygraph{{\C(A \ox B,D)}
    :[r(2.8)] {\C(SD,S(A \ox B))} ^-S
    :[r(3.4)] {\C(I,S(A \ox B))} ^-{\C(n_I,1)}} \\
\xygraph{
    :[r(2.5)] {\C(A,S(A \ox B) \ox A)} ^-{- \ox A}
    :[r(3.5)] {\C(A,SB)} ^-{\C(1,e_{A,B})}}.
\end{multline*}

\item A left inverse for $S:\C(B,SI) \ra \C(SSI,SB)$ is the composite
\[
\xygraph{{\C(SSI,SB)}
    :[r(2.6)] {\C(I,SB)} ^-{\C(n_I,1)}
    :[r(2.2)] {\C(B,SI)} ^-{\varpi_{I,B}}}.
\]
\end{enumerate}
\end{corollary}

\begin{proof}
\begin{enumerate}[{\upshape (i)}]
\item From Proposition~\ref{prop1}, we have
\begin{align*}
\varpi_{I,A \ox B} &= \omega_{I,A \ox B,I} \\
    & = \omega_{A,B,I} \circ \omega_{I,A,B} \\
    &= \varpi_{A,B} \circ \omega_{I,A,B}.
\end{align*}
So the result follows from Proposition~\ref{prop3} and the formula for
$\omega_{I,A,B}$ in terms of $e_{A,B}$.

\item This is a reinterpretation of Proposition~\ref{prop3}.
\end{enumerate}
\end{proof}

We also introduce the natural family of morphisms
\[
\pi_{A,B,C} =
\Big(\xygraph{{\C(A \ox B,C)}
    :[r(2.7)] {\C(SC,S(A \ox B))} ^-S
    :[r(3.3)] {\C(SC \ox A,SB)} ^-{\omega_{SC,A,B}}}\Big).
\]
Taking $C = A \ox B$ and evaluating at the identity, we obtain the natural
transformation
\[
e_{A,B}:S(A \ox B) \ox A \ra SB.
\]
By Yoneda, it follows that $\pi_{A,B,C}$ is the composite
\[
\xygraph{{\C(A \ox B,C)}
    :[r(3.8)] {\C(SC \ox A,S(A \ox B) \ox A)} ^-{S(-) \ox A}
    :[r(4)] {\C(SC \ox A,SB)} ^-{\C(1,e_{A,B})}}.
\]

\begin{proposition}
The natural transformation $\pi$ is invertible if and only if $S$ is fully
faithful.
\end{proposition}

\begin{proof}
If $S$ is fully faithful then each $\pi_{A,B,C}$ is invertible (from the
definition, using invertibility of $\omega_{SC,A,B}$). Conversely, if $\pi$
is invertible, we may take $A = I$ in the definition of $\pi_{A,B,C}$ to
obtain $S:\C(B,C) \ra \C(SC,SB)$ which is consequently invertible. So $S$ is
fully faithful.
\end{proof}

A \emph{right internal hom} for objects $A$ and $E$ of $\C$ is a
representing object $E^A$ (or $[A,E]_r$) for the functor $\C(A \ox
-,E):\C^\op \ra \Set$.

\begin{corollary}\label{cor1.6}
If $S$ is fully faithful then $(SB)^{SC} \cong \lraise{B}{C}$.
\end{corollary}

\begin{proof}
The representability of $\C(SC \ox -,SB)$ by $\lraise{B}{C}$ is guaranteed
by the invertibility of $\pi$.
\end{proof}

An object $E$ is called a {\itshape right raiser} when there exists
a choice of $E^{A}$ for all $A\in \C$.

\begin{proposition}\label{prop-leftright}
A left raiser $D$ is a right raiser if and only if the functor
\[
    S:\C^\op \ra \C
\]
has a left adjoint $S':\C^\op \ra \C$.
\end{proposition}

\begin{proof}
To say $S$ has a left adjoint means that, for each object $A$, there is an
object $S' A$ and an isomorphism $\C(A,SB) \cong \C(B,S' A)$, natural in
$B$. However, we have the natural isomorphism $\C(A,SB) \cong \C(A \ox
B,D)$, and therefore $S' A \cong D^A$.
\end{proof}

Notice that the existence of a family of ``commutativity'' isomorphisms
$c_{A,B}:A \ox B \cong B \ox A$ in $\C$, which only need to be natural in
one of the indices $A$ or $B$, implies that every left raiser $D$ is
automatically also a right raiser; moreover, $D^A \cong \lraise{A}{D}$. This is
the case when $\C$ is braided, or, \emph{a fortiori}, symmetric.

We call $D$ a \emph{raiser} when it is both a left and right raiser. In
this case, the unit and counit for the adjunction $S' \dashv S$ are natural
families of morphisms $\alpha_A:A \ra SS' A$ and $\beta_B:B \ra S'SB$ in
$\C$. (The apparent wrong direction of the counit $\beta$ is explained by
the contravariantness of $S$; Peter Freyd has called this situation ``a
contravariant adjunction on the right''.)

A monoidal category is \emph{left closed} when every object is a left
raiser. It is \emph{closed} when every object is a raiser.

Following Chapter 12 of~\cite{Sbook} we call the object $D$ \emph{left
dualizing} when it is a raiser and each $\alpha_A: A \ra SS'A$ is
invertible. By Proposition~\ref{prop-leftright}, this is equivalent to
requiring $D$ to be a left raiser for which $S$ has a fully faithful left
adjoint. We call $D$ \emph{dualizing} when it is a left raiser and $S$ is an
equivalence. Since an equivalence has a fully faithful left adjoint, it
follows that $D$ is also a right raiser.

A monoidal category is \emph{left $*$-autonomous} when it is equipped with
a left dualizing object. It is \emph{$*$-autonomous}~\cite{B} when it is
equipped with a dualizing object.

Each left $*$-autonomous category is left closed since
\[
    \lraise{B}{A} \cong \lraise{B}{(SS'A)} \cong S (B \ox S'A).
\]
Each $*$-autonomous category is closed since it is left $*$-autonomous
and so left closed, and (by looking at $\C$ with the reversed tensor
product) has right internal hom defined by
\[
    B^C \cong S'(SB \ox C).
\]

\section{Monoidal comonads}\label{sec-mc}

A monoidal comonad on a monoidal category $\C$ consists of a comonad
$G = (G,\delta,\epsilon)$ on $\C$ such that $G:\C \ra \C$
is a monoidal functor and $\delta:G \ra GG$ and $\epsilon:G \ra 1_\C$ are
monoidal natural transformations. So, apart from the comonad axioms, we have
a natural transformation
\[
    \phi_{A,B}:GA \ox GB \ra G(A \ox B)
\]
and a morphism $\phi_0:I \ra GI$ satisfying the following conditions
(where we continue to write as if $\C$ were strict monoidal).
\[
\vcenter{\xymatrix@R=10ex@C=8ex{
GA \ox GB \ox GC \ar[r]^-{1 \ox \phi_{B,C}} \ar[d]_{\phi_{A,B} \ox 1}
    & GA \ox G(B \ox C) \ar[d]^-{\phi_{A,B \ox C}} \\
    G(A \ox B) \ox GC \ar[r]^-{\phi_{A \ox B,C}} & G(A \ox B \ox C)}}
\qquad
\vcenter{\xymatrix@R=4.5ex@C=2ex{
    & GA \ar[dl]_-{\phi_0 \ox 1} \ar[dr]^-{1 \ox \phi_0} \ar[dd]^-1 \\ 
    GI \ox GA \ar[dr]_-{\phi_{I,A}} && GA \ox GI \ar[dl]^-{\phi_{A,I}} \\
    & GA}}
\]
\[
\vcenter{\xygraph{{GA \ox GB}
   (:[r(2.8)] {G(A \ox B)} ^{\phi_{A,B}}
    :[dr] {GG(A \ox B)}="e" ^<(0.4)\delta
    )
    :[dd] {GGA \ox GGB} _-{\delta \ox \delta}
    :[r(2.8)] {G(GA \ox GB)} ^-{\phi_{GA,GB}}
    :"e" _<(0.5){G \phi_{A,B}}}}
~
\vcenter{\xymatrix@R=7ex@C=0ex{
    GA \ox GB \ar[rr]^-{\phi_{A,B}} \ar[dr]_-{\epsilon \ox \epsilon}
    && G(A \ox B) \ar[dl]^-\epsilon \\ & A \ox B}}
\]
\[
\vcenter{\xymatrix@=7ex{
    I \ar[d]_-{\phi_0} \ar[r]^-{\phi_0} & GI \ar[d]^-\delta \\
    GI \ar[r]^-{G\phi_0} & G^2I}}
\qquad\qquad
\vcenter{\xymatrix@R=7ex{
    & GI \ar[dr]^-{\epsilon_I} \\ I \ar[ur]^-{\phi_0} \ar[rr]^-1 && I}}
\]

Let $\C^G$ denote the category of Eilenberg-Moore coalgebras for the comonad
$G$. Objects are pairs $(A,\gamma:A \ra GA)$, called $G$-coalgebras,
satisfying
\[
\vcenter{\xygraph{{A}="s"
    :[d(1.2)] {GA} _-\gamma
    :[r(1.5)] {G^2 A}="t" ^-{\delta_A}
 "s":[r(1.5)] {GA} ^-\gamma
    :"t" ^{G\gamma}}}
\qquad \text{and} \qquad
\vcenter{\xymatrix{
A \ar[dr]_-1 \ar[r]^-\gamma & GA \ar[d]^-{\epsilon_A} \\ & A.}}
\]
Morphisms $f:(A,\gamma) \ra (B,\gamma)$ in $\C^G$ are morphisms $f:A \ra B$
in $\C$ such that the square
\[
\xygraph{{A}="s"
    :[r(1.5)] {GA} ^-\gamma
    :[d(1.2)] {GB}="t" ^-{Gf}
 "s":[d(1.2)] {B} _-f
    :"t" ^-\gamma}
\]
commutes.

We make a note of the following fact:

\begin{proposition}\label{prop-gtalg}
If $G:\C \ra \C$ is a comonad with a left adjoint $T:\C \ra \C$, then $T$
becomes a monad and $\C^G \cong \C^T$. Furthermore, if $G$ is a monoidal
comonad then $T$ is an opmonoidal monad.
\end{proposition}

It is well known~\cite{M} that, if $G$ is a monoidal comonad, then $\C^G$
becomes monoidal in such a way that the underlying functor $U:\C^G \ra \C$
becomes strict monoidal. The tensor product for $\C^G$ is defined by
\[
(A,\gamma) \ox (B,\gamma) = \Big(A \ox B,
\xymatrix@1@C=3em{{A \ox B} \ar[r]^-{\gamma \ox \gamma} & {GA \ox GB} 
    \ar[r]^-{\phi_{A,B}} & {G(A \ox B)}}\Big)
\]
and the unit object is $(I,\phi_0)$.

In the dual setting of opmonoidal monads, the paper of A.~Brugui\`eres and
A.~Virelizier~\cite{BV} provides the structure on the monad in order for
left (or right) autonomy of $\C$ to lift to the category of
Eilenberg-Moore algebras. Here we are interested in lifting $*$-autonomy from
$\C$ to $\C^G$. We begin with structure weaker than $*$-autonomy.

Assume we merely have a functor $S:\C^\op \ra \C$ that we would like to lift
to $\C^G$.
\[
\xymatrix@=3em{(\C^G)^\op \ar@{.>}[r]^-S \ar[d]_{U^\op} & \C^G \ar[d]^-U \\
\C^\op \ar[r]^-S & \C}
\]
By~\cite{S}, we require a $G$-coaction on $SU^\op$; that is, a natural
transformation
\[
\hat{\nu}:SU^\op \ra GSU^\op
\]
satisfying two conditions. Since $U$ has a right adjoint $R$ defined by $RA
= (GA,\delta_A)$, such $\hat{\nu}$ are in bijection with natural
transformations
\[
    \nu:S \ra GSG^\op
\]
(where we use the fact that $G =UR$) satisfying:

\begin{equation}\tag{\text{Axiom 1}}
\vcenter{\xymatrix@C=3ex@R=5ex{ & GSG \ar[dr]^-{\epsilon_{SG}} \\ 
    S \ar[ur]^-{\nu} \ar[rr]_{S \epsilon} && SG}}
\end{equation}

\begin{equation}\tag{\text{Axiom~2}}
\vcenter{\xygraph{{S} 
   (:[r(1.9)] {GSG} ^-{\nu} 
    :[r(1.1)d(0.9)] {GGSG}="e" ^-{\delta_{SG}}
   )
    :[d(1.8)] {GSG} _-{\nu} 
    :[r(1.9)] {GGSGG} ^-{G \nu_G}
    :"e" _-{GGS\delta}}}
\end{equation}

\begin{proposition}
If $S:\C^\op \ra \C$ is a functor and $\nu:S \ra GSG$ is a natural
transformation satisfying Axioms~1 and~2 then a functor $\bar{S}:
(\C^G)^\op \ra \C^G$ is defined by
\[
    \bar{S}(A,\gamma) = (SA,GS\gamma \o \nu_A), \qquad \bar{S}f = Sf.
\]
\end{proposition}

Now suppose we also have a natural transformation
\[
\omega_{A,B,C}:\C(A,S(B \ox C)) \ra \C(A \ox B,SC).
\]
By Yoneda, such natural transformations are in bijection with natural
transformations
\[
e_{B,C}:S(B \ox C) \ox B \ra SC.
\]
The bijection is determined by $e_{B,C} =
\omega_{S(B \ox C),B,C}(1_{S(B \ox C)})$ and $\omega_{A,B,C}$ is the
composite
\[
\xygraph{{\C(A,S(B \ox C))}
    :[r(3.6)] {\C(A \ox B,S(B \ox C) \ox B)} ^-{- \ox B}
    :[r(3.9)] {\C(A \ox B,SC)} ^-{\C(1,e_{B,C})}}.
\]

As we shall see, the condition that $e$ is a $G$-coalgebra morphism is
encapsulated in the following axiom.
\begin{equation}\tag{\text{Axiom~3}}
\hfill \vcenter{\xygraph{{S(A \ox B) \ox GA}
   (:[d(1.2)l(0.4)] {GSG(A \ox B) \ox GGA} _-{\nu_{A \ox B} \ox \delta_A}
    :[d(1.2)r(0.4)] {G(SG(A \ox B) \ox GA)} _-{\phi_{SG(A \ox B),GA}}
    :[r(4.8)] {G(S(GA \ox GB) \ox GA)} _-{G(S\phi_{A,B} \ox 1)}
    :[u(1.2)r(0.4)] {GSGB}="e" _-{Ge_{GA,GB}}
    )
    :[r(2.8)] {S(A \ox B) \ox A} ^-{1 \ox \epsilon_A}
    :[r(2)] {SB} ^-{e_{A,B}}
    :"e" ^-{\nu_B}}}
\end{equation}


\begin{proposition}\label{prop-e-coalg}
Assuming Axioms~1,~2,~and~3, the morphism $e_{A,B}$ becomes a $G$-coalgebra
morphism
\[
    e_{A,B}:\bar{S}((A,\gamma) \ox (B,\gamma)) \ox (A,\gamma) \ra 
    \bar{S}(B,\gamma)
\]
for $G$-coalgebras $(A,\gamma)$ and $(B,\gamma)$. Conversely, if $e_{X,Y}$
is a $G$-coalgebra morphism when $X = (GA,\delta_A)$ and $Y = (GB,\delta_B)$
are cofree G-coalgebras, then Axiom 3 holds.
\end{proposition}

\begin{proof}
The following diagram commutes.
\[
\scalebox{0.9}{
\xygraph{{S(A \ox B) \ox A}="s"
    :[r(2.5)] {S(A \ox B) \ox GA} ^-{1 \ox \gamma}
    (:@<3pt>[r(3)] {S(A \ox B) \ox GGA}="1" ^-{1 \ox \delta_A}
    )
     :@<-3pt>"1" _-{1 \ox G\gamma} 
    :[r(1.5)d(1.2)] {GSG(A \ox B) \ox GGA}="2" ^-{\nu_{A \ox B} \ox 1}
    :[d(1.4)] {G(SG(A \ox B) \ox GA)}="6" ^-\phi
    :[d(1.5)] {G(S(GA \ox GB) \ox GA)}="7" ^-{G(S\phi \ox 1)}
    :[r(0.7)d(1)] {GSGB}="8" ^<(0.6){Ge_{GA,GB}}
    :[d(1.7)l] {GSB}="e" ^-{GS\gamma}
 "s":[d(1.2)l] {GSG(A \ox B) \ox A}="3" _-{\nu_{A \ox B} \ox 1}
    :[d(1.4)] {GS(GA \ox GB) \ox A} _-{GS\phi \ox 1}
    :[d(1.4)] {GS(A \ox B) \ox A} _-{GS(\gamma \ox \gamma) \ox 1}
    :[d(1.4)] {GS(A \ox B) \ox GA}="5" _-{1 \ox \gamma}
    :[d(1.4)r] {G(S(A \ox B) \ox A)}="9" _-\phi
    :"e" ^-{Ge_{A,B}}
 "3":[r(4)] {GSG(A \ox B) \ox GA}="4" ^-{1 \ox \gamma}
    :"2" ^-{1 \ox G\gamma}
 "4":[l(1.5)d(2)] {GS(GA \ox GB) \ox GA} _-{GS\phi \ox 1}
    (:"5" |<(0.55){GS(\gamma \ox \gamma) \ox 1}
    ):[r(1)d(0.9)] {G(S(GA \ox GB) \ox A)}="10" ^-\phi
    :"9" _<(0.35){G(S(\gamma \ox \gamma) \ox 1)}
 "4":[r(0.6)d(1.4)] {G(SG(A \ox B) \ox A)} ^-\phi
   (:"6" ^-{G(1 \ox \gamma)}
   ):"10" ^-{G(S\phi \ox 1)}
    :"7" _-{G(1 \ox \gamma)}
    :[d(1.3)l] {G(S(GA \ox B) \ox GA)}="11" _-{G(S(1\ox \gamma) \ox 1)} 
    :"e" _-{Ge_{GA,B}}
"10":[r(0.3)d(1.3)] {G(S(GA \ox B) \ox A)}="12" ^-{G(S(1 \ox \gamma) \ox 1)}
   (:"9" ^<(0.3){\quad G(S(\gamma \ox 1) \ox 1)}
   ):"11" ^-{G(1 \ox \gamma)}}
}
\]
By Axiom~3, the top route around this diagram is
\[
\scalebox{1}{
\xygraph{{S(A \ox B) \ox A}="s"
    :[r(1.3)u] {S(A \ox B) \ox GA} ^-{1 \ox \gamma}
    :[r(1.3)d] {S(A \ox B) \ox A}="t" ^-{1 \ox \epsilon_A}
 "s":"t" ^-1
    :[r(2.2)] {SB} ^-{e_{A,B}}
    :[r(1.6)] {GSGB} ^-{\nu_B}
    :[r(1.6)] {GSB} ^-{GS\gamma}}
}
\]
and, therefore, the following diagram commutes
\[
\xygraph{{S(A \ox B) \ox A}="s"
    :[d(1.2)] {GSG(A \ox B) \ox A} _-{\nu_{A \ox B} \ox 1}
    :[d(1.2)]{GS(GA \ox GB) \ox A} _-{GS \phi \ox 1}
    :[r(3.6)] {GS(A \ox B) \ox A} ^-{GS(\gamma \ox \gamma) \ox 1}
    :[r(3)] {GS(A \ox B) \ox GA} ^-{1 \ox \gamma}
    :[u(1.2)] {G(S(A \ox B) \ox A)} _-\phi
    :[u(1.2)] {GSB}="e" _-{Ge_{A,B}}
 "s":[r(2.5)] {SB} ^-{e_{A,B}}
    :[r(2.5)] {GSGB} ^-{\nu_B}
    :"e" ^-{GS \gamma}
}
\]
which is precisely the condition for $e_{A,B}$ to be a $G$-coalgebra morphism.

To prove the converse statement we observe that the diagram
\[
\xygraph{{S(A \ox B) \ox GA}="1"
    [r(3.5)u(1.2)] {S(A \ox B) \ox A}="2"
    [r(3)d(1.2)] {SB}="3"
    [r(1)d(1.8)] {GSGB}="4"
 "2"[d(1.2)] {S(GA \ox B) \ox GA}="9"
 "1"[d(1.6)] {GSG(A \ox B) \ox GGA}="8"
    [d(3.2)] {G(SG(A \ox B) \ox GA)}="7"
 "9"[d(1.2)] {S(GA \ox GB) \ox GB}="10"
    [d(1.2)] {GSG(GA \ox GB) \ox GGB}="11"
    [d(1.2)] {G(SG(GA \ox GB) \ox GB)}="12"
    [d(1.2)] {G(S(GGA \ox GGB) \ox GB)}="13"
    [d(1.2)] {G(S(GA \ox GB) \ox GA)}="6"
 "3"[d(1.2)] {SGB}="14"
    [d(2.4)] {GSGGB}="15"
    [d(2.4)] {GSGB}="5"
    "1":"2"     ^-{1 \ox \epsilon}
    "1":"9"     ^-{S(\epsilon \ox 1) \ox 1}
    "1":"8"     _-{\nu \ox \delta}
    "9":"10"    ^-{S(1 \ox \epsilon) \ox 1}
    "9":"3"     ^-e
    "10":"14"   ^-e
    "10":"11"   ^-{\nu \ox \delta}
    "8":"11"    ^-{GSG(\epsilon \ox \epsilon) \ox 1}
    "8":"7"     _-\phi
    "11":"12"   ^-\phi
    "7":"12"    ^-{G(SG(\epsilon \ox \epsilon) \ox 1)\quad}
    "7":@/_3ex/"6" _-{G(S\phi \ox 1)}
    "12":"13"   ^-{G(S\phi \ox 1)}
    "13":@<1ex>"6" ^-{G(S(\delta \ox \delta) \ox 1)}
    "6":@<1ex>"13" ^-{G(S(G\epsilon \ox G\epsilon) \ox 1)}
    "6":@(dl,dr)"6" _-1
    "6":"5"     _-{Ge}
    "2":"3"     ^-e
    "3":"4"     ^-\nu
    "3":"14"    _-{S\epsilon}
    "14":"15"   _-\nu
    "15":"5"    _-{GS\delta}
    "4":"15"    |-{GSG\epsilon}
    "4":@/^3ex/"5" ^-1
    "10"[r(1.8)d(1.8)] {\scriptstyle (\dagger)}}
\]
commutes. The outside of the diagram is Axiom 3 and the region labelled by
$(\dagger)$ exactly expresses that $e_{GA,GB}$ is a $G$-coalgebra morphism.
\end{proof}

\begin{corollary}
If $A$, $B$, and $C$ are $G$-coalgebras then there is a natural
transformation $\omega_{A,B,C}$ such that the following square commutes.
\[
\xymatrix@R=7ex@C=12ex{
\C^G(A,\bar{S}(B \ox C)) \ar@{^{(}->}[d] \ar[r]^-{\omega_{A,B,C}}
    & \C^G(A \ox B,\bar{S}C) \ar@{^{(}->}[d] \\
    \C(UA,S(UB \ox UC)) \ar[r]^-{\omega_{UA,UB,UC}} & \C(UA \ox UB,SUC)}
\]
\end{corollary}

The condition that a morphism
\[
    n_I:I \ra SSI
\]
should be a $G$-coalgebra morphism is given by
\begin{equation}\tag{\text{Axiom~4}}
\vcenter{\xygraph{{I}
   (:[d(1.3)] {SSI} _-{n_I}
    :[r(1.6)] {GSGSI} ^-{\nu_{SI}}
    :[r(2.4)] {GSGSGI.} ^-{GSGS\phi_0}
    :[u(1.3)] {GSSI}="e" _-{GS\nu_I}
    )
    :[r(2)] {GI} ^-{\phi_0}
    :"e" ^-{Gn_I}}}
\end{equation}

\begin{theorem}
Suppose $D$ is a left raiser in the monoidal category $\C$ and define $S$,
$e$, and $n$ as in Section~\ref{sec-raisers}. Suppose $(G,\delta,\epsilon)$
is a monoidal comonad on $\C$. If $\nu:S \ra GSG$ is a natural
transformation satisfying Axioms~1, 2, 3,~and~4 then $(D,GS\phi_0 \o \nu_I)$ is
a left raiser in $\C^G$.
\end{theorem}

\begin{proof}
By Proposition~\ref{prop-e-coalg} and the fact that $e_B = e_{B,I}$, we have
that $e_B:\bar{S}(B,\gamma) \ox (B,\gamma) \ra \bar{S}(I,\phi_0)$ is a
$G$-coalgebra morphism. By Axiom~4 and Corollary~\ref{cor2}(iii), we have
that $n_A:(I,\phi_0) \ra \bar{S}(\bar{S}(A,\gamma) \ox (A,\gamma))$ is a
$G$-coalgebra morphism. From the formula for $\rho_{A,B}$ in terms of $n$,
we see that $\rho_{A,B}$ restricts as follows:
\[
\vcenter{\xygraph{{\C^G(A,\bar{S}B)}="s"
    :[r(3.3)] {\C^G(I,\bar{S}(A \ox B))} ^-{\rho_{A,B}}
    :@{^{(}->}[d(1.2)] {\C(I,S(UA \ox UB)).}="t"
 "s":@{^{(}->}[d(1.2)] {\C(UA,SUB)} 
     :"t" ^-{\rho_{UA,UB}}}}
\]
Since $\rho_{A,B} = \omega^{-1}_{I,A,B}$, it follows from
Corollary~\ref{cor2}(i) that $\omega_{I,A,B}:\C^G(I,\bar{S}(A \ox B)) \ra
\C^G(A,\bar{S}B)$ is invertible. By Proposition~\ref{prop1}, it follows that
\[
    \omega_{A,B,C}:\C^G(A,\bar{S}(B \ox C)) \ra \C^G(A \ox B,\bar{S}C)
\]
is invertible. Taking $C = (I,\phi_0)$, we have that $\bar{S}(I,\phi_0) =
(D,GS\phi_0 \o \nu_I)$ is a left raiser in $\C^G$, as required.
\end{proof}

In other words, we have
\[
\xymatrix@R=7ex@C=10ex{
\C^G(A \ox B,C) \ar@{^{(}->}[d] \ar[r]^-{\pi_{A,B,C}}
    & \C^G(\bar{S}C \ox A,\bar{S}B) \ar@{^{(}->}[d] \\
    \C(UA \ox UB,UC) \ar[r]^-{\pi_{UA,UB,UC}} & \C(SUC \ox UA,SUB)}
\]
for $A,B,C \in \C^G$. It follows that if the $\pi$ for $\C$ is injective,
then so is the $\pi$ for $\C^G$.

\section{Star-autonomous monoidal comonads}\label{sec-starmc}

Suppose $G$ is a monoidal comonad on the monoidal category $\C$. Suppose $S:
\C^\op \ra \C$ has a left adjoint $S':\C \ra \C^\op$ with unit $\alpha:1
\ra SS'$ and counit $\beta:1 \ra S'S$. Suppose we have $\nu:S \ra GSG$ and
$\nu':S' \ra GS'G$ each satisfying Axioms~1 and~2 so that we obtain liftings
\[
\bar{S}: (\C^G)^\op \ra \C^G \quad \mathrm{and} \quad
\bar{S}':\C^G \ra (\C^G)^\op.
\]

Consider the following conditions:

\begin{equation}\tag{\text{Axiom 5}}
\vcenter{\xygraph{{G}="0"
    :[r(2)] {GSS'}="1" ^-{G\alpha}
 "0":[d(1.4)] {SS'G} _-{\alpha_G}
    :[r(2)] {GSGS'G} ^-{\nu_{S'G}}
    :"1" _-{GS \nu'}}}
\end{equation}

\begin{equation}\tag{\text{Axiom 6}}
\vcenter{\xygraph{{G}="0"
    :[r(2)] {GS'S}="1" ^-{G\beta}
 "0":[d(1.4)] {S'SG} _-{\beta_G}
    :[r(2)] {GS'GSG} ^-{\nu'_{S'G}}
    :"1" _-{GS' \nu}}}
\end{equation}

\begin{proposition}
The unit $\alpha_A:(A,\gamma) \ra \bar{S}\bar{S}'(A,\gamma)$ is a
$G$-coalgebra morphism for each $G$-coalgebra $\gamma:A \ra GA$ if and only
if Axiom~5 holds. The counit $\beta_A:(A,\gamma) \ra
\bar{S}'\bar{S}(A,\gamma)$ is a $G$-coalgebra morphism for each
$G$-coalgebra $\gamma:A \ra GA$ if and only if Axiom~6 holds. Consequently,
if Axioms~5 and~6 hold, then $\bar{S}'$ is a left adjoint for $\bar{S}$.
\end{proposition}

\begin{proof}
The outside of the following diagram expresses that $\alpha_A:(A,\gamma) \ra
\bar{S}\bar{S}'(A,\gamma)$ is a $G$-coalgebra morphism.
\[
\xygraph{{A}="11"
    :[r(3)] {GA}="12" ^-\gamma
    :[d(1.2)r] {GSS'A}="23" ^-{G\alpha_A}
"11":[d(1.2)] {SS'A}="21" _-{\alpha_A}
    :[d(1.2)] {GSGS'A}="31" _-{\nu_{S'A}}
    :[r(3)] {GSGS'GA}="32" _-{GSGS' \gamma}
    :"23" _-{GS \nu'_A}
"21":[r(2)] {SS'GA}="22" ^-{SS' \gamma}
    :"32" _-{\nu_{S'GA}}
"12":"22" _-{\alpha_{GA}}
"22"[r] {\scriptstyle (\dagger)}}
\]
The two unlabelled regions commute by naturality and the region labelled by
$(\dagger)$ is simply Axiom~5. Conversely, for all objects $B$ of $\C$, taking
$A = GB$ and $\gamma = \delta_B$ in the $G$-coalgebra condition, we obtain
the commutativity of region labelled by $(\ddagger)$ in the following diagram.
\[
\xygraph{{GB}="11"
    :[r(3)] {GGB}="12" ^-{\delta_B}
    :[r(3)] {GB}="13" ^-{G \epsilon_B}
    :[d(1.2)] {GSS'B}="23" ^-{G \alpha_B}
"11":[d(1.2)] {SS'GB}="21" _-{\alpha_{GB}}
    :[d(1.2)] {GSGS'GB}="31" _-{\nu_{S'GB}}
    :[r(3)] {GSGS'GGB}="32" _-{GSGS'\delta_B}
    :[r(3)] {GSGS'GB}="33" _-{GSGS'G\epsilon_B}
    :"23" _-{GS\nu'_B}
"12":[d(1.2)] {GSS'GB}="22" ^-{G\alpha_{GB}}
    :"23" ^-{GSS' \epsilon_B}
"32":"22" _-{GS\nu'_{GB}}
"21"[r(1.4)] {\scriptstyle (\ddagger)}}
\]
The unlabelled squares in the diagram commute by functoriality and
naturality, so the outside of the square commutes and is seen to
be Axiom~5.

The second sentence is dealt with symmetrically.
\end{proof}

\begin{definition}
A monoidal comonad $G$ on a left $*$-autonomous monoidal category $\C$ is
\emph{left $*$-autonomous} when it is equipped with natural transformations
$\nu:S \ra GSG$ and $\nu':S' \ra GS'G$, each satisfying Axioms 1 and 2, with
$\nu$ satisfying Axioms~3 and~4, and together satisfying Axioms~5 and~6. If
$\C$ is $*$-autonomous, we also say $G$ is $*$-autonomous when it is left
$*$-autonomous.
\end{definition}

Although we obtain the following corollary as an immediate consequence
of our results, it is the desired object of this paper.

\begin{corollary}
If $\C$ is a left $*$-autonomous monoidal category and $G$ is a left
$*$-autonomous monoidal comonad on $\C$ then the monoidal category $\C^G$
is left $*$-autonomous and the strict monoidal underlying functor $U:\C^G
\ra \C$ preserves left internal homs. If $\C$ is $*$-autonomous then so is
$\C^G$ and $U$ preserves left and right internal homs.
\end{corollary}

\section{Monoidal adjunctions and monoidal comonads} \label{sec-moncom}

Now we step back a bit and work in the reverse direction. Suppose $U \dashv
R:\C \ra \A$ with unit $\eta:1 \ra RU$ and counit $\epsilon:UR \ra 1$, such
that $U$ is strong monoidal and the square
\[
\xygraph{{\A^\op}="s"
    :[r(1.6)] {\C^\op} ^-U
    :[d(1.2)] {\C}="t" ^-S
 "s":[d(1.2)] {\A} _-S
    :"t" ^-U}
\]
commutes. Then we obtain a monoidal comonad $G = UR$ on $\C$ as
\begin{align*}
\delta &= U \eta_R : UR \ra URUR~, \\
\epsilon &= \epsilon: UR \ra 1~, \\
\phi &= \xygraph{{\big(\,URA \ox URB}
        :[r(2.5)] {U(RA \ox RB)} ^-\phi
        :[r(3.2)] {URU(RA \ox RB)} ^-{U \eta_{RA \ox RB}}} \\
       & \hspace{18ex} \xygraph{:[r(2.5)] {UR(URA \ox URB)} ^-{UR\phi^{-1}}
        :[r(3.6)] {UR(A \ox B)\,\big)} ^-{UR(\epsilon_A \ox \epsilon_B)}}, \\
\phi_0 &= \big(\xygraph{{I}
                :[r(1.1)] {UI} ^-{\phi_0}
                :[r(1.6)] {URUI} ^-{U \eta_I}
                :[r(1.9)] {URI} ^-{UR\phi_0^{-1}}}\big)~.
\end{align*}
There is also a candidate for $\nu:S \ra GSG$, viz,
\begin{equation}\label{nu-ur}\tag{$\bigstar$}
\xygraph{{S}
    :[r(1.8)] {SUR=USR} ^-{S\epsilon}
    :[r(3.2)] {URUSR=URSUR} ^-{U \eta_{SR}}}.
\end{equation}

\begin{proposition}\label{prop-nu-ur}
Axioms~1 and~2 hold for the data above.
\end{proposition}

\begin{proof}
The following diagrams respectively show that Axioms~1~and~2 are
satisfied.
\[
\xygraph{{S}
    :[r(1.3)] {SUR} ^-{S\epsilon}
    :[r(1.5)] {USR} ^-=
    (:[r(2)] {URUSR} ^-{U\eta_{SR}}
        (:[d(1.2)] {USR}="u" ^-{\epsilon_{USR}}
        )
         :[r(1.8)] {URSUR} ^-=
         :[d(1.2)] {SUR}="e" ^-{\epsilon_{SUR}}
    ):"u" {USR}="v" _-1
     :"e" ^-=}
\]
\[
\xygraph{{URSUR}="1" 
    :[d(1.3)] {URSURUR}="2" _-{URS\epsilon}
    :[d(1.3)] {URUSRUR}="3" _-=
    :[d(1.3)] {URURUSRUR}="4" _-{URU\eta}
    :[r(3.2)] {URURUSR}="5" _-{URURUS\eta}
    :[r(2.3)] {URURSUR}="6" _-=
    :[u(1.3)] {URUSR}="8" _-{UR\epsilon}
    :[u(1.3)] {URURUSR}="9" _-{U\eta}
 "2":[r(3.2)] {URSUR}="11" ^-{URSU\eta}
    :[d(1.3)] {URUSR} ^-=
    (:"5" ^-{URU\eta}
    ):"8" ^-1
 "1":[r(5.5)] {URURSUR}="10" ^-{U\eta}
 "1":"11" ^-1
 "9":"10" _-=
}
\]
\end{proof}

Now suppose $\A$ and $\C$ are equipped with natural transformations
$e_{A,B}:S(A \ox B) \ox A \ra SB$ and that $U$ preserves these; that is, the
following diagram commutes.
\[
\vcenter{\xygraph{{US(A \ox B) \ox UA}
   (:[u(1.1)r] {U(S(A \ox B) \ox A)} ^-\phi _-\cong
    :[r(3.4)] {USB} ^-{Ue_{A,B}}
    :[d(1.1)r] {SUB.}="e" ^-=
    )
   :[d(1.1)r] {SU(A \ox B) \ox UA} _-=
   :[r(3.4)] {S(UA \ox UB) \ox UA} _-{S\phi \ox 1}
   :"e" _-{e_{UA,UB}}}}
\]

\begin{proposition}\label{adj-ax3}
In this case, Axiom~3 holds.
\end{proposition}

\begin{proof}
This leads us to examine the diagram in Figure~\ref{fig-axiom3}. It can be
seen that the unlabelled areas of the diagram commute. The area labelled by
(C) commutes by the above assumption and the area labelled by (A) is seen to
commute by examining the following diagram.
\[
\xygraph{{USR(A \ox B) \ox URA}="1"
    :[r(3)u] {URUSR(A \ox B) \ox URURA}="2" ^-{U\eta \ox U\eta}
    :[r(3)d] {URSUR(A \ox B) \ox URURA}="3" ^-=
    :[d(1.3)] {U(RSUR(A \ox B) \ox RURA)}="4" ^-\phi
    :[d(1.3)] {URU(RSUR(A \ox B) \ox RURA)}="5" ^-{U\eta}
    :[d(1.3)] {UR(URSUR(A \ox B) \ox URURA)}="6" ^-{UR\phi^{-1}}
    :[d(1.3)l] {UR(SUR(A \ox B) \ox URA)}="7" ^-{UR(\epsilon \ox \epsilon)}
 "1":[d(1.3)] {U(SR(A \ox B) \ox RA)}="8" _-\phi
    :[d(1.3)] {URU(SR(A \ox B) \ox RA)}="9" _-{U\eta}
    :[d(1.3)] {UR(USR(A \ox B) \ox URA)}="10" _-{UR\phi^{-1}}
    :[d(1.3)r] {UR(USR(A \ox B) \ox URA)}="11" _-{UR\phi^{-1}}
 "2":[d(1.3)] {U(RUSR(A \ox B) \ox RURA)}="12" ^-{U\eta \ox U\eta}
    (:"4" ^-=
    ):[d(1.3)] {URU(RUSR(A \ox B) \ox RURA)}="13" ^-{U\eta}
    (:"5" ^-=
    ):[d(1.3)] {UR(URUSR(A \ox B) \ox URURA)}="14" ^-{U\phi^{-1}}
    (:"6" ^-=
    ):"11" ^-{UR(\epsilon \ox \epsilon)}
 "8":"12" |-{U(\eta \ox \eta)}
 "9":"13" |-{URU(\eta \ox \eta)}
"10":"14" |-{UR(U\eta \ox U\eta)}
"11":"7" ^-=}
\]
To see that the region labelled by (B) commutes observe that the following
diagram commutes.
\[
\scalebox{0.9}{
\xygraph{{S(URA \ox URB) \ox URA}="1"
    :[d(1.2)] {SUR(URA \ox URB) \ox URA}="2" _-{S\epsilon \ox 1}
    :[d(1.2)] {USR(URA \ox URB) \ox URA}="3" _-=
    :[d(1.2)] {U(SR(URA \ox URB) \ox RA)}="4" _-{\phi}
    :[r(3.8)d(0.8)] {U(SRU(RA \ox RB) \ox RA)}="5"
    :[r(3.6)u(0.8)] {U(S(RA \ox RB) \ox RA)}="6"
 "1":[r(3.8)u(0.8)] {SU(RA \ox RB) \ox URB}="7" ^-{S\phi^{-1} \ox 1}
    :[r(3.6)d(0.8)] {US(RA \ox RB) \ox URB}="8" ^-=
 "2":[r(3.8)u(0.8)] {SURU(RA \ox RB) \ox URB}="9" ^-{SUR\phi^{-1} \ox 1}
    :[r(3.6)d(0.8)] {SU(RA \ox RB) \ox URB}="10" ^-{SU\eta}
 "3":[r(3.8)d(0.8)] {USRU(RA \ox RB) \ox URB}="11" ^-{USR\phi^{-1} \ox 1}
    :[r(3.6)u(0.8)] {US(RA \ox RB) \ox URB}="12" ^-{US\eta}
 "4":"5" _<(0.2){U(SR\phi^{-1} \ox 1)\quad}
 "5":"6" _<(0.4){\quad U(S\eta \ox 1)}
 "7":"9" ^-{S\epsilon \ox 1}
 "8":"10" ^-=
 "9":"11" ^-=
"10":"12" ^-=
"11":"5" ^-\phi
"12":"6" ^-\phi}
}
\]
By our assumption the $U$ preserves $e$, the upper route of the above
diagram is
\[
\xygraph{{S(URA \ox URB) \ox URA}
    :[r(2.6)] {SURB} ^-e
    :[r(1.5)] {USRB}="3" ^-1
    [r(2.7)] {U(S(RA \ox RB) \ox RA)}
    :"3" _-{Ue}}
\]
which then shows the commutativity of the region labelled by (B).
\end{proof}

\begin{sidewaysfigure}
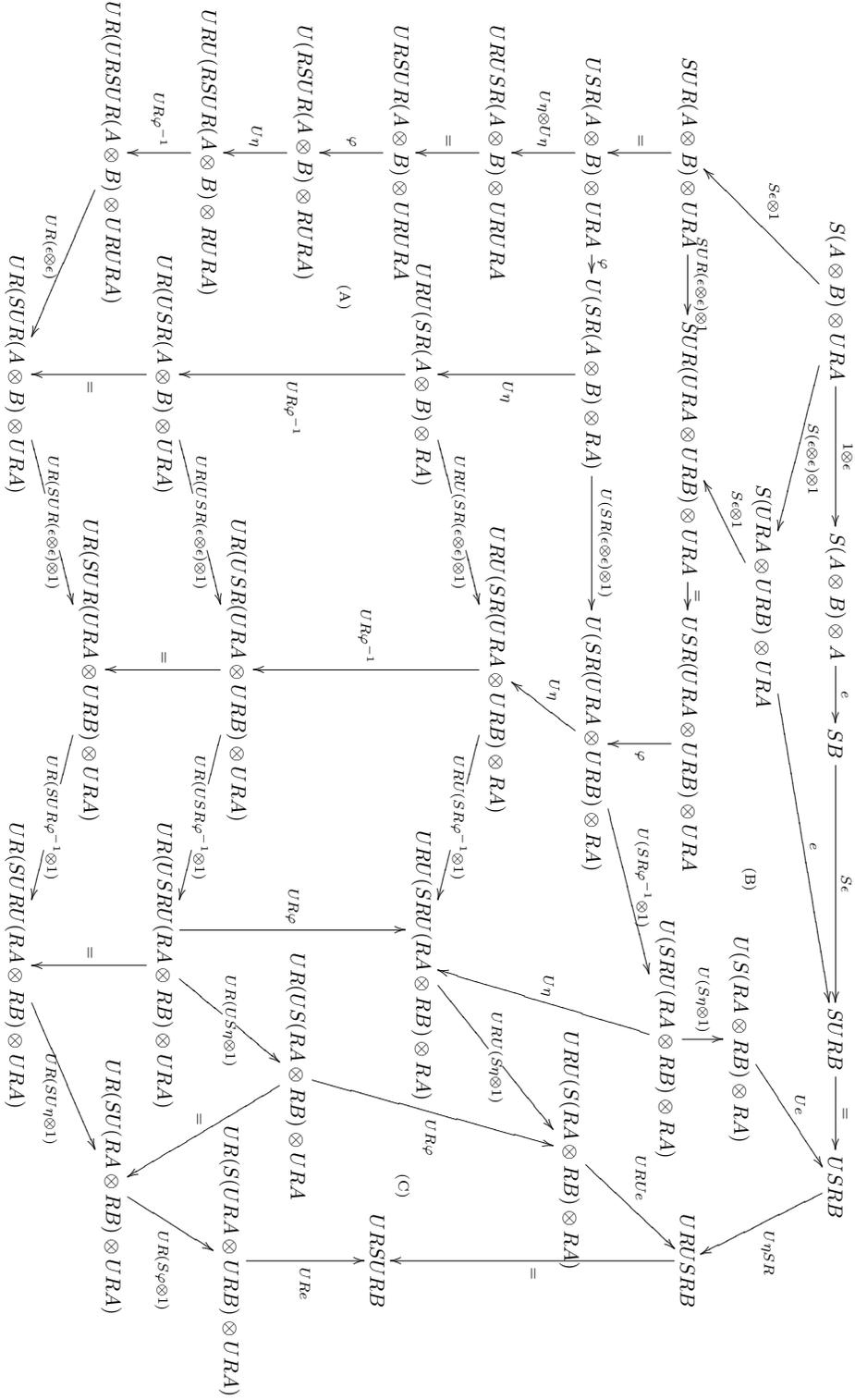
\centering
\[
\scalebox{0.84}{
\xygraph{{S(A \ox B) \ox URA}="1"
    :[d(2)l(2)] {SUR(A \ox B) \ox URA}="2" _-{S \epsilon \ox 1}
    :[d(1.3)] {USR(A \ox B) \ox URA}="3" _-=
    :[d(1.3)] {URUSR(A \ox B) \ox URURA}="4" _-{U\eta \ox U\eta}
    :[d(1.3)] {URSUR(A \ox B) \ox URURA}="5" _-=
    :[d(1.3)] {U(RSUR(A \ox B) \ox RURA)}="6" _-\phi
    :[d(1.3)] {URU(RSUR(A \ox B) \ox RURA)}="7" _-{U\eta}
    :[d(1.3)] {UR(URSUR(A \ox B) \ox URURA)}="8" _-{UR\phi^{-1}}
    :[r(3)d(1.3)] {UR(SUR(A \ox B) \ox URA)}="9" _-{UR(\epsilon \ox \epsilon)}
    :[r(4)u] {UR(SUR(URA \ox URB) \ox URA)}="10" |-{UR(SUR(\epsilon \ox
                                                  \epsilon) \ox 1)}
    :[r(4)d] {UR(SURU(RA \ox RB) \ox URA)}="11" |-{UR(SUR\phi^{-1} \ox 1)}
    :[r(3)u(1.3)] {UR(SU(RA \ox RB) \ox URA)}="12" _-{UR(SU\eta \ox 1)}
    :[u(1.6)r] {UR(S(URA \ox URB) \ox URA)}="13" _-{UR(S\phi \ox 1)}
    :[u(2)] {URSURB}="14" _-{URe}
 "1":[r(4)] {S(A \ox B) \ox A}="24" ^-{1 \ox \epsilon}
    :[r(2)] {SB}="25" ^-{e}
    :[r(4)] {SURB}="26" ^-{S\epsilon}
    :[r(2)] {USRB}="27" ^-=
    :[r(1)d(2)] {URUSRB}="28" ^-{U\eta SR}
    :"14" ^-=
 "1":[r(4)d] {S(URA \ox URB) \ox URA}="44" ^-{S(\epsilon \ox \epsilon) \ox 1}
    :"26" ^-e
 "2":[r(4)] {SUR(URA \ox URB) \ox URA}="45" ^-{SUR(\epsilon \ox \epsilon) \ox 1}
    :[r(4)] {USR(URA \ox URB) \ox URA} ^-=
    :[d(1.3)] {U(SR(URA\ox URB)\ox RA)}="21" ^-\phi
"44":"45" _-{S\epsilon \ox 1}
 "3":[r(3)] {U(SR(A \ox B) \ox RA)}="16" ^-\phi
    :[d(2.3)] {URU(SR(A \ox B) \ox RA)}="17" ^-{U\eta}
    :[d(3.5)] {UR(USR(A \ox B) \ox URA)}="18" ^-{UR\phi^{-1}}
    :"9" ^-=
"16":"21"  ^-{U(SR(\epsilon \ox \epsilon)\ox 1)}
    :[r(4)u] {U(SRU(RA \ox RB) \ox RA)}="22" ^-{U(SR\phi^{-1} \ox 1)}
    :[u] {U(S(RA \ox RB) \ox RA)}="23" ^-{U(S\eta \ox 1)}
    :"27" ^-{Ue}
"17":[r(4)u] {URU(SR(URA \ox URB) \ox RA)}="29" |-{URU(SR(\epsilon \ox
                                                  \epsilon) \ox 1)}
    :[r(4)d] {URU(SRU(RA \ox RB) \ox RA)}="30" |-{URU(SR\phi^{-1} \ox 1)}
    :[r(2.5)u(2)] {URU(S(RA \ox RB) \ox RA)}="31" |-{URU(S\eta \ox 1)}
"21":"29" _-{U\eta}
"22":"30" _-{U\eta}
"31":"28" ^-{URUe}
"18":[r(4)u] {UR(USR(URA \ox URB) \ox URA)}="19" |-{UR(USR(\epsilon \ox
                                                    \epsilon) \ox 1)}
    (:"10" _-=
    ):[r(4)d] {UR(USRU(RA \ox RB) \ox URA)}="20" |-{UR(USR\phi^{-1} \ox 1)}
    (:"11" _-=
    ):[u(1.8)r(1.5)] {UR(US(RA \ox RB) \ox URA}="32" |-{UR(US\eta \ox 1)}
"29":"19" _-{UR\phi^{-1}}
"32":"12" _-=
"32":"31" _-{UR\phi}
"20":@<4ex>"30" ^-{UR\phi}
 "8":@<2ex>@{}"17" |<(0.7){\text{(A)}}
"21":@<4ex>@{}"26" |<(0.5){\text{(B)}}
"13":@<2ex>@{}"31" |<(0.5){\text{(C)}}}
}
\]
\caption{Proof of Proposition~\ref{adj-ax3}}\label{fig-axiom3}
\end{sidewaysfigure}

\begin{remark}
By Yoneda, Axiom~3 is equivalent to the commutativity of 
\[
\vcenter{\xygraph{{\C(A \ox B,C)}="s"
    :[r(2.8)] {\C(SC \ox A,SB)} ^-\pi
    :[r(4)] {\C(SC \ox GA,GSGB)}="e" ^-{\C(1 \ox \epsilon_A,\nu_B)}
 "s":[d(1.2)] {\C(G(A \ox B),GC)} _-G
    :[d(1.2)] {\C(GA \ox GB,GC)} _-{\C(\phi_{A,B},1)}
    :[r(3.1)] {\C(SGC \ox GA,SGB)} ^-\pi
    :[r(3.7)] {\C(G(SGC \ox GA),GSGB).} ^-G
    :[u(1.2)] {\C(GSGC \ox G^2A,GSGB)} _-{\C(\phi_{SGC,GA},1)}
    :"e" _-{\C(\nu_C \ox \delta_A,1)}}}
\]
\end{remark}

\begin{proposition}
The formula for $\nu$ given in~\eqref{nu-ur} recovers the original
$\nu$ when applied to the adjunction $U \dashv R:\C \ra \C^G$ in the
setting of Proposition~\ref{prop-nu-ur}.
\end{proposition}

\begin{proof}
We have that $U(A,\gamma) = A$, $RX = (GX,\delta_X)$, $\epsilon:UR \ra 1$ is
just the counit $\epsilon$ of the comonad, and $\eta:1 \ra RU$ has components
$\gamma:(A,\gamma) \ra (GA,\delta_A)$. So the $\nu$ given in~\eqref{nu-ur}
becomes
\[
\xygraph{{SX}="s"
    :[d(1.5)] {SGX} _-{S\epsilon_X}
    :[r(2)] {GSG^2 X}="w" _-{\nu_{GX}}
    :[r(2.2)] {GSGX~.}="e" _-{GS\delta_X}
 "s":[r(2)] {GSGX} ^-{\nu_X}
    (:"w" _-{GSG \epsilon_X})
    :"e" ^-1}
\]
\end{proof}

Now suppose $\A$ and $\C$ are equipped with morphisms $n_I:I \ra SSI$
and that $U$ preserves this; that is, the following diagram commutes.
\[
\vcenter{\xygraph{{I}="0"
    :[d(1.2)] {UI} _-{\phi_0}
    :[r(2.8)] {USSI} ^-{U n_I}
    :[u(1.2)] {SSUI}="1" _-=
 "0":[r(1.2)] {SSI} ^-{n_I}
    :"1" ^-{SS \phi}}}
\]

\begin{proposition}\label{adj-ax4}
In this case, Axiom~4 holds.
\end{proposition}

\begin{sidewaysfigure}
\[
\scalebox{0.95}{
\xymatrix@R=3em{
    {I} \ar[rr]^-{n_I} \ar[d]_-{\phi_0} \ar@{}[drr]|{\text{(A)}} &&
    S^2 I \ar[d]^-{S^2 \phi_0} \ar[r]^-{S \epsilon} &
    {SURSI} \ar[d]^-{SURS\phi_0} \ar[r]^-= &
    {USRSI} \ar[d]^{USRS \phi_0} \ar[r]^-{U \eta} &
    {URUSRSI} \ar[d]^{URUSRS \phi_0} \ar[r]^-= &
    {URSURSI} \ar[d] ^-{URSURS \phi_0}
\\
    UI \ar[d]_-{U\eta} \ar[r]^-{Un_I} &
    US^2I \ar[r]^-= \ar[d]^-{U\eta} &
    S^2UI \ar[d]^-{S^2U\eta} \ar[r]^-{S \epsilon} &
    SURSUI \ar[d]^-{SURSU\eta} \ar[r]^-= &
    USRSUI \ar[d]^-{USRSU\eta} \ar[r]^-{U\eta} &
    URUSRSUI \ar[d]^-{URUSRSU\eta} \ar[r]^-= &
    URSURSUI \ar[d]^-{URSURSU\eta}
\\
    URUI \ar[dd]_{UR\phi^{-1}_0} \ar[r]^{URUn_I} \ar@{}[ddr]|{\text{(A)}} &
    URUS^2I \ar[d]^-= &
    S^2URUI \ar[d]^-{S^2UR\phi^{-1}_0} \ar[r]^-{S \epsilon} &
    SURSURUI \ar[d]^-{SURSUR\phi^{-1}_0} \ar[r]^-= &
    USRSURUI \ar[d]^-{USRSUR\phi^{-1}_0} \ar[r]^-{U\eta} &
    URUSRSURUI \ar[d]^-{URUSRSUR\phi^{-1}_0} \ar[r]^-= &
    URSURSURUI \ar[d]^-{URSURSUR\phi^{-1}_0}
\\
    &
    URS^2UI \ar[d]^-{URS^2\phi^{-1}_0} &
    S^2URI \ar[d]^-= \ar[r]^-= &
    SURSURI \ar[d]^-= \ar[r]^-= &
    USRSURI \ar[d]^-= \ar[r]^-= &
    URUSRSURI \ar[d]^-= \ar[r]^-= &
    URSURSURI \ar[d]^-=
\\
    URI \ar[r]^-{URn_I} &
    URS^2I \ar[ddrr]_-1 &
    SUSRI \ar[dr]^-1 \ar[r]^-{S\epsilon} \ar@{}[d]|{\text{(B)}}&
    SURUSRI \ar[d]^-{SU\eta} \ar[r]^-= &
    USRUSRI \ar[d]^-{US\eta} \ar[r]^-{U\eta} &
    URUSRUSRI \ar[d]^-{URUS\eta} \ar[r]^-= & 
    URSURUSRI \ar[d]^-{URSU\eta}
\\
    &&&
    SUSRI \ar[r]^-= &
    US^2RI \ar[r]^-{U\eta} &
    URUS^2RI \ar[d]^-= &
    URSUSRI \ar[dl]^-=
\\
    &&&
    URS^2I &&
    URS^2URI \ar[ll]_-{URS^2\epsilon}}
}
\]
\caption{Proof of Proposition~\ref{adj-ax4}}\label{fig-adj-ax4}
\end{sidewaysfigure}

\begin{proof}
In the diagram in Figure~\ref{fig-adj-ax4} the unlabelled regions are easily
seen to commute, the two regions labelled by (A) commute by our assumption
that $U$ preserves $n_I$, and the region labelled by (B) is seen to commute
from the following diagram.
\[
\xygraph{{US^2I}="0"
    :[r(2.5)] {S^2UI}="1" ^-= 
    :[r(2.5)] {S^2URUI}="2" ^-{S^2U\eta}
    :[d] {S^2UI}="3" ^-{S^2UR\phi^{-1}_0}
    :[d] {US^2I}="4" ^-=
    :[d] {URUS^2I}="5" ^-{U\eta}
    :[d] {URS^2UI}="6" ^-=
    :[l(2.5)d] {URS^2I}="7" ^-=
 "0":[r(2.5)d] {US^2RUI}="8" _-{US^2\eta}
 "8":"2" _-=
 "8":"4" |-{US^2R\phi^{-1}_0}
 "8":[d] {URUS^2RUI}="9" _-{U\eta}
    :"5" |-{URUS^2R\phi^{-1}_0}
 "9":[d] {URS^2URUI}="10" _-=
    :[d] {URS^2UI}="11" ^-{URS^2\epsilon}
    :"7" _-{URS^2\phi^{-1}_0}
 "0":[dd] {URUS^2I}="13" _-{U\eta}
    :[d] {URS^2UI}="14" _-=
    :"11" _-1
"13":"9" ^-{URUS^2\eta}
"14":"10" ^-{URS^2U\eta}
}
\]
\end{proof}

Finally, suppose that in $\A$ and $\C$ we have $S' \dashv S$ with unit $1
\ra SS'$ and counit $1 \ra S'S$ and that $U$ preserves these, meaning both
\[
    \vcenter{\xygraph{{U}="s"
        :[r(1.6)] {USS'} ^-{U \alpha}
        :[d(1.2)] {SS'U}="t" ^-=
     "s":"t" _-{\alpha_U}}}
\quad\qquad \text{and} \quad\qquad
    \vcenter{\xygraph{{U}="s"
        :[r(1.6)] {US'S} ^-{U \beta}
        :[d(1.2)] {S'SU}="t" ^-=
     "s":"t" _-{\beta_U}}}
\]
commute.

\begin{proposition}
In this case, both Axioms~5 and~6 hold.
\end{proposition}

\begin{proof}
The commutativity of the following diagram proves that Axiom~5 holds and
Axiom~6 is proved with a similar diagram.
\[
\xygraph{{G=UR}="a"
    [r(2.4)] {USS'R}="b"
    [d(1)r(2.8)] {URUSS'R}="c"
 "a"[u(1)r(2.4)] {URUR}="d"
    [d(1)r(2.8)] {URSS'UR}="e"
    [d(1)r(2)] {URSS'}="f"
 "d"[r(2.5)] {UR}="g"
 "a":"d" ^-{U\eta_R} 
 "a":"b" ^-{U\alpha_R}
 "b":"c" _-{U\eta_{SS'R}}
 "c":"e" _-=
 "d":"c" |-{URU\alpha_R} 
 "d":"e" ^-{UR\alpha_{UR}} 
 "d":"g" ^-{UR\epsilon} 
 "e":"f" _-{URSS'\epsilon} 
 "g":"f" ^-{UR\alpha} 
 "a"[d] {SS'UR}="1"
    [d] {SURS'UR}="2"
    [d] {USRS'UR}="3"
    [d] {URUSRS'UR}="4"
    [d] {URSURS'UR}="5"
 "b"[d] {SUS'R}="6"
    [d] {SURUS'R}="7"
    [d] {USRUS'R}="8"
    [d] {URUSRUS'R}="9"
    [d] {URSURUS'R}="10"
 "c"[d] {SUS'R}="11"
    [d] {USS'R}="12"
    [d] {URUSS'R}="13"
    [d] {URSUS'R}="14"
 "f"[d(2)] {URSS'UR}="15"
 "a":"1" _-{\alpha_G}
 "b":"6" _-=
 "1":"6" ^-=
 "1":"2" _-{S\epsilon_{S'UR}}
 "2":"7" ^-=
 "2":"3" _-=
 "3":"8" ^-=
 "3":"4" _-{U\eta_{SRS'UR}}
 "4":"9" ^-=
 "4":"5" _-=
 "5":"10" _-=
 "6":"11" ^-1
 "6":"7" _-{S\epsilon_{US'R}}
 "7":"11" ^-{SU\eta_{S'R}}
 "7":"8" _-= 
 "8":"12" ^-{US\eta_{S'R}}
 "8":"9" _-{U\eta_{SRUS'R}}
 "9":"13" ^-{URUS\eta_{S'R}}
 "9":"10" _-=
 "10":"14" _-{URSU\eta_{S'R}}
 "11":"12" _-=
 "12":"13" _-{U\eta_{SS'R}}
 "13":"14" _-=
 "13":"15" ^-=
 "14":"15" _-=
 "15":"f" _-{URSS'\epsilon}
}
\]
\end{proof}

\section{The Hopf algebra example} \label{sec-hopfeg}

In this section we show that any Hopf algebra with bijective antipode $H$
in a braided
$*$-autonomous category $\C$ gives rise to a $*$-autonomous monoidal comonad
$G:\C \ra \C$ defined on objects by $GX = \lraise{H}{X}$. A left adjoint for
this $G$ is given by $T = H \ox -$ and so, by Proposition~\ref{prop-gtalg},
there is a bijection between $G$-coalgebras $X \ra X^H$ and $T$-algebras
$H \ox X \ra X$. As tensors are easier to work with here than homs, we will
take this latter view in the remainder of this section and show that $T$ is a
$*$-autonomous opmonoidal monad.

For the functor $T = H \ox -:\C \ra \C$, the data of a $*$-autonomous monad
is given as follows.
\begin{align*}
\mu &= \xygraph{{H \ox H \ox X} :[r(2.2)] {H \ox X} ^-{\mu \ox 1}} \\
\eta &= \xygraph{{X} :[r(1.6)] {H \ox X} ^-{\eta \ox 1}} \\
\psi &= \big(\xygraph{{H \ox X \ox Y}
        :[r(3)] {H \ox H \ox X \ox Y} ^-{\delta \ox 1 \ox 1}
        :[r(3.1)] {H \ox X \ox H \ox Y} ^-{1 \ox c \ox 1}}\big) \\
\psi_0 &= \xygraph{{H} :[r(1.1)] {I} ^-{\epsilon}} \\
\nu &= \big(\xygraph{{H \ox S(H \ox X)}
        :[r(2.6)] {S(H \ox X) \ox H} ^-c
        :[r(2.8)] {S(H \ox X) \ox H} ^-{1 \ox \nu}
        :[r(2.1)] {SX} ^-e}\big) \\
\nu' &= \big(\xygraph{{H \ox S'(H \ox X)}
        :[r(3.7)] {H \ox S'(X \ox H)} ^-{\nu^{-1} \ox S'c^{-1}}
        :[r(2.1)] {S'X} ^-{e'}}\big)
\end{align*}

That $T$ forms an opmonoidal monad is standard and so in the remainder of
this section we concentrate on showing that this data satisfies the axioms
of a $*$-autonomous monad.

\subsection{Lifting the internal hom}

We begin with Axiom~3. The remainder of the axioms are straightforward and
will be proved in Section~\ref{subsec-remainder}. Axiom~3 is meant to express
that $e_{A,B}:S(A \ox B) \ox A \ra SB$ is a morphism of $T$-algebras, i.e., a
morphism of left $H$-modules. We will prove this statement directly, instead
of proving Axiom~3. It follows from the more general statement in
Proposition~\ref{prop-ealg} below.

In this section we need only assume that $\C$ is a braided closed category.
Suppose $H$ is a Hopf algebra in $\C$ and that $M,N \in \C$ are left
$H$-modules. Denote by $e_{M,N}:\MN \ox M \ra N$ the morphism obtained by
evaluating the isomorphism $\C(L,{}^M N) \cong \C(L \ox M,N)$ at the identity.

Note that we have the composites
\begin{align*}
& \xygraph{{H \ox \MN \ox M}
    :[r(2.4)] {H \ox N} ^-{1 \ox e}
    :[r(1.4)] {N} ^-{\mu}} \\
& \xygraph{{\MN \ox H \ox M}
    :[r(2.4)] {\MN \ox M} ^-{1 \ox \mu}
    :[r(1.5)] {N} ^-{e}} \\
\end{align*}
which, under the isomorphism $\C(L \ox M,N) \cong \C(L,{}^M N)$, become
respectively left and right actions of $H$ on $\MN$:
\begin{align*}
\mu_l &:H \ox \MN \ra \MN \\
\mu_r &:\MN \ox H \ra \MN
\end{align*}
It is not too difficult to see that these actions make $\MN$ into a
$H$-bimodule or a left $H \ox H^\op$-module. Restriction of scalars along
the algebra morphism
\[
\xygraph{{H}
    :[r(1.4)] {H \ox H} ^-\delta 
    :[r(2)] {H \ox H^\op} ^-{1 \ox \nu}}
\]
makes $\MN$ into a left $H$-module.

\begin{proposition}\label{prop-ealg}
The evaluation morphism $e_{M,N}: \MN \ox M \ra N$ is a morphism of
$H$-modules.
\end{proposition}

\begin{proof}
By the definition of $\mu_l$ and $\mu_r$ the following two diagrams commute.
\[
\xygraph{{H \ox \MN \ox M}="1"
    [d(1.2)] {\MN \ox M}="2"
    [r(2.4)] {N}="3"
    [u(1.2)] {H \ox N}="4" 
 "1":"2" _-{\mu_l \ox 1}
 "2":"3" ^-{e}
 "1":"4" ^-{1 \ox e}
 "4":"3" ^-{1 \ox e}}
\qquad
\xygraph{{\MN \ox H \ox M}="1"
    [d(1.2)] {\MN \ox M}="2"
    [r(2.4)] {N}="3"
    [u(1.2)] {\MN \ox M}="4" 
 "1":"2" _-{\mu_r \ox 1}
 "2":"3" ^-{e}
 "1":"4" ^-{1 \ox \mu}
 "4":"3" ^-{e}}
\]
Using these facts it is possible to see that the diagram (where we have
dropped the ``$\ox$'')
\[
\scalebox{1}{
\xygraph{{H ~ \MN ~ M}="1"
    [d(2)] {H ~ H ~ H ~ \MN ~ M}="2"
    [d(2)] {H ~ \MN ~ H ~ H ~ M}="3"
    [r(4)] {H ~ \MN ~ H ~ H ~ M}="4"
    [r(4)] {\MN ~ H ~ M}="5"
    [u(2)] {\MN ~ M}="6"
    [u(2)] {M}="7"
 "1"[d(1)r(1.5)]{H ~ H ~ \MN ~ M}="8"
    [d(2)r(0.5)]{H ~ \MN ~ H ~ M}="9"
 "4"[u(2)r(0)]{H ~ \MN ~ H ~ M}="11"
    [u(2)l(0)]{H ~ \MN ~ M}="10"
    [r(2)]{H ~ M}="13"
 "5"[u(2)l(1.5)]{\MN ~ M}="12"
"1":"2" _-{\delta^3 ~ 1 ~ 1}
"1":"8" ^-{\delta ~ 1 ~ 1}
"1":"10" ^-1
"2":"3" _-{1 ~ c_{HH,\MN} ~ 1}
"3":"4" _-{1 ~ 1 ~ \nu ~ 1 ~ 1}
"4":"5" _-{\mu_l ~ 1 ~ \mu}
"4":"11" _-{1 ~ 1 ~ \mu ~ 1}
"5":"6" _-{\mu_r ~ 1}
"5":"12" ^-{1 ~ \mu}
"6":"7" _-{e}
"8":"9" _-{1 ~ c ~ 1}
"9":"3" _-{1 ~ 1 ~ \delta ~ 1}
"9":"10" ^-{1 ~ 1 ~ \epsilon ~ 1}
"10":"11" ^-{1 ~ 1 ~ \eta ~ 1}
"10":"12" ^-{\mu_l ~ 1}
"10":"13" ^-{1 ~ e}
"11":"12" ^-{\mu_l ~ \mu}
"13":"7" ^-{\mu}}}
\]
commutes and expresses that $e_{M,N}$ is a morphism of $H$-modules.
\end{proof}

That $e_{A,B}:S(A \ox B) \ox A \ra SB$ becomes a morphism of $H$-modules
follows by choosing $M=A$ and $N= {}^B \! D$.

\subsection{The remainder of the axioms} \label{subsec-remainder}

It is left to show that the remainder of the axioms hold. Axioms~1 and~2 may
respectively be expressed in terms of a monad as follows:
\[
    \vcenter{\xygraph{{ST}="1"
        [r(2)] {S}="2"
    "1"[dr] {TST}="3"
    "1":"2" ^-{S\eta}
    "1":"3" _-{\eta_{ST}}
    "3":"2" _-\nu}}
\qquad\qquad
    \vcenter{\xygraph{{TTST}="1"
        [d(1.2)] {TTSTT}="2"
        [r(3)] {TST}="3"
        [u(1.2)] {S}="4"
    "1"[r(1.7)] {TST}="5"
    "1":"5" ^-{\mu_{ST}}
    "1":"2" _-{TTS\mu}
    "2":"3" ^-{T\nu_{T}}
    "3":"4" _-\nu
    "5":"4" ^-\nu}}
\]
The diagram
\[
\xygraph{{S(H \ox X)}="1"
    [r(3)] {H \ox S(H \ox X)}="2"
    [d(1.2)] {S(H \ox X) \ox H}="3"
    [d(1.2)] {S(H \ox X) \ox H}="4"
    [l(3)] {SX}="5"
 "1":"2" ^-{\eta \ox 1}
 "1":"3" ^-{1 \ox \eta}
 "1":"4" _-{1 \ox \eta}
 "1":"5" _-{S(\eta \ox 1)}
 "2":"3" ^-c
 "3":"4" ^-{1 \ox \nu}
 "4":"5" _-e}
\]
shows that Axioms~1 holds and Axiom~2 is seen to be satisfied since the
following diagram commutes (we have dropped the ``$\ox$'' symbol).
\[
\xygraph{{H \> H \> S(H \> X)}="1"
    [d(1.2)] {H \> H \> S(H \> H \> X)}="2"
    [d(1.2)] {H \> S(H \> H \> X) \> H}="3"
    [d(2.4)] {H \> S(H \> H \> X) \> H}="4"
    [r(2.7)] {H \> S(H \> X)}="5"
    [r(2.4)] {S(H \> X) \> H}="6"
    [r(2.4)] {S(H \> X) \> H}="7"
    [u(1.2)] {SX}="8"
 "1"[r(3)] {S(H \> X) \> H \> H}="9"
    [r(4.5)] {S(H \> X) \> H }="10"
    [d(2.4)] {S(H \> X) \> H}="12"
 "9"[d(1.2)] {S(H \> H \> X) \> H \> H}="13"
    [d(1.2)] {S(H \> H \> X) \> H \> H}="14"
    [d(1.2)] {S(H \> H \> X) \> H \> H}="15"
 "9"[d(1.2)r(2.5)] {S(HX) HH}="11"
    [d(1.2)] {S(H \> X) \> H \> H}="16"
    [d(1.2)] {S(H \> H \> X) \> H \> H}="17"
"1":"9" ^-{c_{HH,S(HX)}}
"1":"2"_-{1 \> 1 \> S(\mu \> 1)}
"2":"13"^-{c_{HH,S(HHX)}}
"2":"3"_-{1 \> c}
"3":"14"^-{c_{H,S(HHX)X}}
"3":"4"_-{1 \> 1 \> \nu}
"4":"15"|-{c_{H,S(HHX)H}}
"4":"5"_-{1 \> e}
"5":"6"_-{c}
"6":"7"_-{1 \> \nu}
"7":"8"_-{e}
"9":"13" _-{S(\mu \> 1) \> 1 \> 1}
"9":"10" ^-{1 \> \mu}
"9":"11" ^-{1 \> c}
"10":"12" ^-{1 \> \nu}
"11":"16" ^-{1 \> \nu \> \nu}
"12":"8" ^-{e}
"13":"14" _-{1 \> c}
"14":"15" _-{1 \> \nu \> 1}
"15":"6" ^-{e \> 1}
"15":"17" ^-{1 \> 1 \> \nu}
"16":"12" ^-{1 \> \mu}
"16":"17" ^-{S(\mu \> 1) \> 1 \> 1}
"17":"7" ^-{e \> 1}
"17":"8" ^-{e}}
\]
Similar diagrams show that $S'$ and $\nu'$ also satisfy Axioms~1 and~2.

Axiom~4 expressed in terms of a monad is
\[
\xygraph{{TI}="1"
    [d(1.2)] {TSSI}="2"
    [r(2)] {TSTSTI}="3"
    [r(2.4)] {TSTSI,}="4"
    [u(1.2)] {SSI}="5"
 "1"[r(2.2)] {I}="6"
"1":"6" ^-{\psi_0}
"1":"2" _-{Tn_I}
"2":"3" ^-{TS\nu_I}
"3":"4" ^-{TSTS\psi_0}
"4":"5" ^-{\nu_{SI}}
"6":"5" ^-{n_I}}
\]
and that this holds may be seen from the following diagram (where we have
again dropped the ``$\ox$'').
\[
\xygraph{{H}="1"
    [r(2.25)]{I}="15"
 "1"[d(1.2)]{H \, SSI}="2"
    [d(1.8)l(1.6)]{H \, S(SH \, H)}="3"
 "2"[d(2.4)]{H \, S(SH \, H)}="10"
 "2"[d(2.4)r(2.3)]{H \, S(SH \, H)}="11"
"10"[d(1.2)]{H \, S(SH \, H)}="4"
"11"[d(1.2)]{H \, S(SH \, H)}="5"
    [u(1.2)r(2.2)]{H \, S(H \, SI)}="6"
    [u(1.2)r(1.6)]{S(H \, SI) \, H}="7"
    [u(1.2)]{S(H \, SI) \, H}="8"
"11"[u(2.4)]{SSI \, H}="14"
 "6"[u(1.2)]{H \, S(H \, SI)}="12"
    [u(1.2)]{S(H \, SI) \, H}="13"
    [u(1.2)]{SSI}="9"
"1":"2" _-{1 \, n_I}
"1":"15" ^-{\epsilon}
"2":"3" _-{1 \, Se}
"2":"6" |-{1 \, S(\epsilon \, 1)}
"2":"10" ^-{1 \, Se}
"2":"12" |-{1 \, S(\epsilon \, 1)}
"2":"14" ^-{c}
"3":@/_9pt/"4" _-{1 \, S(1 \, \nu)}
"4":"5" _-{1 \, Sc}
"4":"11" _-{1 \, S(S\epsilon \, 1)}
"5":@/_9pt/"6" _-{1 \, S(1 \, S\epsilon)}
"6":@/_7pt/"7" _-{c}
"6":"12" |-{1 \, S(\nu \, 1)}
"7":"8" _-{1 \, \nu}
"7":"13" |-{S(\nu \, 1) \, 1}
"8":@/_7pt/"9" _-{e}
"10":"4" |-{1 \, S(S\nu \, 1)}
"10":"11" ^-{1 \, S(S\epsilon \, 1)}
"11":"6" ^-{1 \, Sc}
"12":"13" _-c
"13":"9" _-e
"14":"13" ^-{S(\epsilon \, 1) \, 1}
"14":"9" ^-{1 \, \epsilon}
"15":"9" ^-{n_I}
}
\]

Axiom~5 may be expressed in terms of a monad as
\[
\xygraph{{T}="1"
    [d(1.2)]{TSS'}="2"
    [r(2)]{TSTS'T.}="3"
    [u(1.2)]{SS'T}="4"
"1":"4" ^-{\alpha_T}
"1":"2" _-{T \alpha}
"2":"3" ^-{TS\nu'}
"3":"4" _-{\nu_{S'T}}}
\]
Note that the commutativity of the diagram
\[
\xygraph{{\C(C,S'(A \ox B))}="1"
    [d(1.2)]{\C(B \ox C,S'A)}="2"
    [r(2.8)]{\C(A,S(B \ox C))}="3"
    [u(1.2)]{\C(A \ox B,SC)}="4"
"1":"4" ^-{\sigma}
"1":"2" _-{\omega'}
"2":"3" _-{\sigma}
"3":"4" _-{\omega}} 
\]
implies that
\[
\xygraph{{A \ox B}="1"
    [d(1.2)]{SS'A \ox B}="2"
    [r(3.2)]{S(B \ox S'(A \ox B)) \ox B}="3"
    [u(1.2)]{SS'(A \ox B)}="4"
"1":"4" ^-\alpha
"1":"2" _-{\alpha \ox 1}
"2":"3" ^-{Se' \ox 1}
"3":"4" _-e}
\]
commutes, and therefore, that the following diagram expressing Axiom~5
commutes.
\[
\scalebox{0.95}{\xygraph{{H \ox X}="1"
    [d(1.8)] {H \ox SS'X}="2"
    [d(1.8)] {H \ox S(H \ox S'(X \ox H))}="3"
    [r(7)] {H \ox S(H \ox S'(H \ox X))}="4"
    [u(1.2)] {S(H \ox S'(H \ox X)) \ox H}="5"
    [u(1.2)] {S(H \ox S'(H \ox X)) \ox H}="6"
    [u(1.2)] {SS'(H \ox X)}="7"
 "1"[r(2)] {X \ox H}="8"
    [r(2)] {SS'(X \ox H)}="9"
 "2"[u(0.5)r(2)] {SS'X \ox H}="10"
 "3"[u(1.1)r(3.5)]{S(X \ox S'(X \ox H)) \ox H}="11"
"1":"2" _-{1 \ox \alpha}
"1":"8" ^-{c}
"2":"3" _-{1 \ox S'e}
"2":"10" ^-{c}
"3":"4" _-{1 \ox S(\nu^{-1} \ox S'c^{-1})}
"3":"11" _-{c}
"4":"5" _-{c}
"5":@<-6pt>"6" _-{1 \ox \nu}
"5":@<0pt>"6" ^-{S(\nu \ox 1) \ox 1}
"6":"7" _-{e}
"8":"9" ^-\alpha
"8":"10" ^-{\alpha \ox 1}
"9":"7" ^-{SS'c^{-1}}
"10":"11" _-{Se' \ox 1}
"11":"6" ^-{S(1 \ox S'c^{-1}) \ox 1 \quad}
"11":"9" ^-{e}}}
\]
Axiom~6 may be handled similarly. Thus $T = H \ox -$ is a $*$-autonomous monad.


\end{document}